\documentclass[conference]{IEEEtran}
\makeatletter
\def\ps@headings{%
\def\@oddhead{\mbox{}\scriptsize\rightmark \hfil \thepage}%
\def\@evenhead{\scriptsize\thepage \hfil \leftmark\mbox{}}%
\def\@oddfoot{}%
\def\@evenfoot{}}
\makeatother 

\usepackage{cite}

 \usepackage{graphicx}

\ifCLASSINFOpdf
\else
\fi
\usepackage[cmex10]{amsmath}
\usepackage{array}

\usepackage{mdwmath}
\usepackage{mdwtab}

\usepackage{eqparbox}

\usepackage[tight,footnotesize]{subfigure}

\usepackage{fixltx2e}

\usepackage{stfloats}

\usepackage{url}

\begin{document}
%
\title{An analytical model for evaluating outage and handover probability of cellular wireless networks}

\author{\IEEEauthorblockN{L. Decreusefond, P. Martins, T. T. Vu}
\IEEEauthorblockA{Institut Telecom\\ Telecom Paristech\\CNRS LTCI\\Paris, France}
}


%


\newtheorem{definition}{Definition}
\newtheorem{theorem}{Theorem}
\newtheorem{proposition}{Proposition}
\newtheorem{assumption}{Assumption}
\newtheorem{lemma}{Lemma}
\newtheorem{collary}{Collary}
\newtheorem{remark}{Remark}

\maketitle



%
\IEEEpeerreviewmaketitle

\begin{abstract}
We consider stochastic cellular networks where base stations locations form a homogenous Poisson point process and each mobile is attached to the base station that provides the best mean signal power. The mobile is in outage if the SINR falls below some threshold. The handover decision has to be made if the mobile is in outage for some time slots. The outage probability and the handover probability is evaluated in taking into account the effect of path loss, shadowing, Rayleigh fast fading, frequency factor reuse and conventional beamforming. The main assumption is that the Rayleigh fast fading changes each time slot while other network components remain static during the period of study.
\end{abstract}

\section{Introduction}
In a wireless network, nodes can be modeled by fixed or
stochastic pattern of points on the plane. Fixed points model can
be finite or infinite and usually regular or lattice. This
approach fails to capture the irregularity and randomness of a
real network. For example, to model a wireless cellular network,
the hexagonal cellular network is the model of this type most
used. In reality, the base station (BS) nodes are usually fixed,
it is not true that they are spastically periodical. Recently, stochastic model of
nodes are more preferred. Node
patterns can be represented by a stochastic process on the plane
such as Poisson point process . It is worth to note that the stochastic
models, although are more complicated at the first sight, usually
lead to elegant and easy calculated formulas. In fact, all
information obtained when studying both types of model are useful
for the design or dimensioning processes of networks. In this
paper we choose the stochastic approach and investigate a cellular
network with homogeneous Poisson point process of BS.

However most of works
relies on the assumption that a mobile once in the network is
served by the nearest BS. This is due to considering path loss
exponent model of radio propagation and remove the effect of
fading. This assumption results to a so called Poisson-Voronoi
cells model (for example, \cite{StochasticGeometryandWirelessNetworksVolumeI}, page 63). Most works consider only the effect of fading
but only the slow fading such as log-normal shadowing or fast
fading such as Rayleigh fading. Besides, most of works consider the well known exponent propagation model. In this paper, the proposed model is sightly more general. Firstly, we consider
a general model of path loss. Secondly, we are interested in a
system spastically static but some temporal evolutionary elements.
More precisely, we include both random general slow fading and
Rayleigh fading but that the slow fading being static in time, and
the Rayleigh fading changes each time slot. Thirdly, once
considering this, we make a very natural assumption that the
mobile is served by the BS that provides the most strong mean signal
power in time (best server). The mean signal power depends on path loss and slow fading. This choice of serving BS
can be made either by the mobile or the operator. Thus, it can be
though that our model is a generalization of Poisson-Voronoi cell
model. If we assume that each BS generates an independent copy of
a continuous shadowing random such as the one in \cite{GaussianRandomField}, one
can interpret continuous cell form. However if each slow fading field generated
by a BS is independent fading fields then the cells can not be
analytically identified, in particular they can not be measurable.
We do not address this issue in this paper. Nevertheless we make
assumption that the slow fading value of BSs to a mobile is
independent, so including all above cases.

Once the mobile is served by one BS, the signal received by this
BS will be the useful signal, and we assume that the considered
system other signal received by other BS using the same frequency
is interference. It is not true if we consider for example an
advanced system in which the base stations are cooperative.
However our model covers almost all existing cellular networks. To
model the frequency reuse, we add a independent mark on our Poisson point process of
BSs. A BS interferes other BSs that have the same mark. In
addition to the interferences, the local noise can intervene. In
order to make communication with the BS, the
signal-to-noise-plus-interferences ratio (SINR) at this mobile
location must be excess some threshold, in this case the mobile is
covered, in contrary it is in outage. If the mobile is in outage
during a period of time, i.e for some consecutive time slots, a
handover decision has to be made. It can be made by the mobile, the
served BS, the network system or even by a neighbor BS. In this
paper, we are interested in the calculation of the outage
probability and the handover probability in explicit
forms. Since we assumes a homogenous Poisson point process of BSs, but not fixed patterns, these results does not depend on the position of the mobile and can be considered as global, meaning on all MS on the network.

This present paper benefits from results in the literature. In \cite{AGeometricInterpretationofFadinginWirelessNetworks}, Haenggi shows that the path loss fading process is Poisson point process in real line in the case of path loss exponent model. In \cite{AnAlohaprotocolformultihopmobilewirelessnetworks},\cite{StochasticGeometryandWirelessNetworksVolumeI} and \cite{StochasticGeometryandWirelessNetworksVolumeII}, Baccelli and al. find analytical expressions for outage probability of networks where each node tries to connect with a destination of fixe distance or the nearest node in case of Rayleigh fading. In \cite{SpatialOutageProbabilityforCellularNetworks}, Kelif and al. find a outage probability expression for cellular network by mean of the so-called fluid model. In \cite{SpatialandTemporalCorrelationoftheInterferenceinALOHAAdHocNetworks}, Ganti and al. find the interesting results about temporal and spatial correlation of wireless networks. In \cite{CIRcumulativedistributioninaregularnetwork} and \cite{ComparisonofVariousFrequencyReusePatternsforWiMAXNetworkswithAdaptiveBeamforming}, outage probability of regular hexagonal cellular networks with reuse factor and adaptive beamforming is studied  by simulation.

This paper is organized as following. In the section \ref{section:
Model scenario} we describe our model. In this section \ref{section: POISSON POINT PROCESS of pathloss shadowing} we show that the path loss shadowing is a Poisson point process in real line. In the section
\ref{Section: Outage proba} we calculate outage probability. In the section \ref{section: Handover analysis} we calculate the handover probability. Section \ref{section: Simulation} shows the numerical results and the difference between our model and the traditional hexagonal model.
\section{System model and scenario} \label{section: Model scenario}

\subsection{Propagation model} \label{subsec: Propagation model}
The signal radio propagation modelling is complicated, usually
divided by a deterministic large scale path loss and the random
fading components. The large scale path loss describes the channel
at a microscopic level. If there are a BS (base station) located
at $y$ and an mobile located at $x$ and the transmission power
$P$, the mobile's received signal has the average power $L(y-x)P$
where $L$ is the path loss function. We assume that $L$ is
measurable function on $R^2$.

The most used path loss function is the path loss exponent law
$L(z)=K|z|^{-\gamma}$ where $|z|$ refers to the Euclid norm of
$z$. The parameter $K$ depends on the frequency, the antenna
height,... while the path loss exponent $\gamma$ characterizes the
environment under study. $\gamma$ is typically in the range of
$(2,4)$, it may be greater if the environment is very dense urban.
In fact, this path loss model is not correct for small distances
and has infinite mean of interference for Poisson patterns of BSs
\cite{AnAlohaprotocolformultihopmobilewirelessnetworks}. To avoid theses problems, one can
use the modified path loss exponent model $L(z)=K(\max\{R_0,|z|\})^{-\gamma}$ where $R_0$ is a reference distance.

In addition to the deterministic large scale effect, there are two
random factors can be considered. The first, called
\emph{shadowing} or \emph{shadow fading}, represents the signal
attenuation caused by a large obstacle such as building. The
second, called \emph{fast fading}, represents the impact of
multipath phenomena, or in other word many objects scatter the
signal. The shadowing can be considered as constant during a
period of communication of the mobile while the fast fading
changes each time slot. If there is no beamforming technique is used, the received signal power from BS $y$ to MS $x$ at the time slot $l$ will be
\begin{equation}
P_{yx}[l] = r_{y,x}[l]h_{yx}L(y-x)P ,
\end{equation}
where $\{h_{yx}\}_{x,y\in R^2}$ are copies of a random variable
$H$ while $\{r_{yx}[l]\}$ are independent copies of $R$ which is an exponential random variable of mean $1/\mu$. We suppose that for
each $x$, the random variables $h_{yx}$ for all $y\in R^2$ are
independent. We define $p_H$ the probability density function of
$H$ and $F_H(\beta)=P(H\geq \beta)=\int_{\beta}^{\infty}p_H(t)dt$.
The most used shadowing random
model is log-normal shadowing, for which $H$ is a log-normal
random variable. In this case, we can write $H\sim 10^{G/10}$
where $G\sim \mathcal{N}(0,\sigma^2)$.

In short word, we suppose that during the period of study, the
shadowing remains constant while  the Rayleigh fast fading changes at each
time slot.

Consider the path loss shadowing process $\Xi = \{\xi_i = (h_{yx}L(y-x)P)^{-1}\}$. We make the following assumptions:

\begin{assumption}\label{assumption: Independent of H}
For each $x\in R^2$, $\{h_{yx}\}_{y\in R^2}$ are independent.
\end{assumption}

\begin{assumption}\label{assumption: Lebegue of LH = 0}
 $H$ admits a continuous probability density function on $(0,\infty)$.
\end{assumption}

\begin{assumption}\label{assumption: Max is obtian}
Define $B(\beta)=\int_{R^2}F_H((L(z)P\beta)^{-1})d z$, then $0<B(\beta)<\infty$ for all $\beta>0$.
\end{assumption}

 By the displacement theorem we will show that $\Xi$ is a simple Poisson point process in the real line $(0,+\infty)$ (proposition \ref{proposition: POISSON POINT PROCESS of path shadow}) of intensity $\Lambda(dt)=\lambda_BB'(\beta)>0$. Hence we have the right to reorder $\Xi$ such that $\xi_0<\xi_1<...$ and for simplicity we do it.
\subsection{Poisson point process of BSs} \label{subsec: POISSON POINT PROCESS of BS}
We assume the homogenous Poisson point process of BSs $\Pi_B = \{y_0,y_1,...\}$ of
intensity $\lambda_B$ on $R^2$, each BS transmit has constant
transmitted power $P$. For any details on Poisson point process we
refer to \cite{StochasticGeometryandWirelessNetworksVolumeI}.
\begin{assumption}\label{assumption: m choose BS}
Once being in the network the mobile $x$ is firstly attached to (or served
by) the BSs that provide the best \emph{average} signal strength
in \emph{time}. In other word, it is attached by $y_0$ (after reordering and renumbering $\Xi$).
\end{assumption}

\subsection{Beamforming model} \label{subsec: Beamforming model}
We consider the conventional beamformer technique with $n_t$ antennas. The power radiation pattern for
a conventional beamformer is a product of array factor and
radiation pattern of a single antenna. If $\phi$ is the look direction (toward which the beam is steered), the array gain in the direction $\theta$ is given by (\cite{ComparisonofVariousFrequencyReusePatternsforWiMAXNetworkswithAdaptiveBeamforming},\cite{CIRcumulativedistributioninaregularnetwork}):
\begin{eqnarray*}
\frac{\sin^2(n_t\frac{\pi}{2}(\sin(\theta)-sin(\phi))}{n_t\sin^2(\frac{\pi}{2}(\sin(\theta)-sin(\phi))}g(\theta) ,
\end{eqnarray*}
where $g(\theta)$ is the gain in the direction $\theta$ with one antenna. For simplicity we assume that the BS always steers to the direction of the serving MS and the gain $g(\theta)$ is positive constant on $(-\pi/2,\pi/2)$ and $0$ otherwise (zero front-to-back power ratio). Hence, the interference signal power from a BS to a MS attached by an other BS using the same frequency in the direction $\theta$ will be reduced by a factor of:
\begin{eqnarray*}
a(\theta)=1_{\{\theta\in(-\pi/2,\pi/2)\}}\frac{\sin^2(n_t\frac{\pi}{2}(\sin(\theta)))}{n_t^2\sin^2(\frac{\pi}{2}(\sin(\theta)))} \cdotp
\end{eqnarray*}
If the beamforming technique is not used, we will simply use $a(\theta)=1$.
\subsection{Frequency reuse}
We add a mark to each BSs
$e_i$. The marks $e_i$ are independent
copies of the random variable $E$ who is uniformly distributed on
$\{1,2,...,k\}$ where $k$ will be called the frequency reuse factor. The BSs that have the same mark interfere between themselves. Our reuse model can be considered as the worst case where the bandwidth is divided into $k$ subband and each BS is randomly attributed a sub band. It is contrast to the hexagonal network pattern where the  interfering BS must be placed far from a reference BS.
\subsection{SINR}

Assume that each other BS using the same frequency is always serving a MS, and the MS $x$ is in the direction $\theta_i$ which is i.i.d chosen on $(-\pi,\pi)$ of the BS $i$. The SINR
at the time slot $l$ is defined as:
\begin{equation}
s_x[l] =
\frac{r_{y_0x}[l]\xi_0^{-1}}{N+\sum_{i\neq
0}1_{\{e_i=e_0\}}a(\theta_i)r_{y_ix}[l] \xi_i^{-1}}  ,
\end{equation}
where $N$ is a constant noise power. The term $I_x=\sum_{i\neq
b(x)}1_{\{e_i=e_0\}}a(\theta_i)r_{y_ix}[l] \xi_i^{-1}$ is the
sum of all interferences. In order to make communication with the
attached BS, the SINR must not fall below some threshold $T$.

\subsection{Handover decision}
We consider a simple SINR based decision. The handover should be made if the MS is in outage for $n$ consecutive time slots.
\subsection{Scenario} \label{subsec: Scenario}
The scenario is as following:
\begin{itemize}
\item Realization of a snapshot of BSs $y_i$, slow
fading $h_{y_ix}$ and the frequency $e_i$.

\item Attachment of mobile $x$ to best BS, ($y_{0}$ after reordering $\Xi$).

\item Realization of the directions $\theta_i$ for interfering BS $y_i$.

\item At time slot $l$, realization of Rayleigh fast fading $r_{y_ix}[l]$ and calculate the SINR $s_x[l]$. If $s_x[l]<T$ then the mobile is in outage otherwise it is covered.

\item If the mobile is in outage for $n$ consecutive time slots
then the handover should be made.
\end{itemize}

The outage probability is then $p_o(T)=P(s_x[l]<T)$ and the
handover decision probability is $p_{ho}(T)=P(s_x[l]<T,...,s_x[l+n-1]<T)$. We
also define the coverage probability $p_C(T)=P(s_x[l]\geq T)$.

\subsection{Interference limited case}
We are particularly interested in the interference limited regime
when the noise power $N$ is negligible or nearly equal to zero as
it happens usually in a real network. We can set $N=0$. The outage
probability calculated in the interference regime can be
considered as an upper bound for the outage probability in the
general case.
\section{Poisson point process of path loss shadowing} \label{section: POISSON POINT PROCESS of pathloss shadowing}

\subsection{General case}
\begin{proposition}\label{proposition: POISSON POINT PROCESS of path shadow}
$\Xi$
is a Poisson point process on $R^{+}=(0,\infty)$ with intensity density
$\Lambda(d t)=\lambda_B B'(t)d t$.

\end{proposition}
\begin{proof}
Define the marked point process $\Pi^x =
\{y_i,h_{y_ix}\}_{i=0}^{\infty}$. It is a Poisson point process of intensity
$\lambda_Bd y\otimes f_H(t)d t$ because the marks are i.i.d. We
consider the probability kernel $p((z,t),A) = 1_{\{(L(z)Pt)^{-1}\in A\}}$
for all Borel $A\in R^{+}$ and apply the displacement theorem
(\cite{StochasticGeometryandWirelessNetworksVolumeI}, theorem 1.3.9) to obtain that the point process
$\Xi$ is Poisson point process
of intensity
\begin{eqnarray*}
\Lambda(A)=\lambda_B\int_{R^2\otimes R}1_{(\{L(z)tP)^{-1}\in A\}}p_H(t)d z
d t  \cdotp
\end{eqnarray*}
We now show that $\Lambda([0,\beta])=\lambda_BB(\beta)$.
Indeed,
\begin{eqnarray*}
\Lambda([0,\beta]) &=& \lambda_B\int_{R^2\otimes
R}1_{\{t\geq (\beta P.L(z)P)^{-1}\}}p_H(t)d z d t\\
        &=&\lambda_B\int_{R^2}F_H((\beta L(z)P)^{-1})d z\\
        &=&\lambda_BB(\beta) \cdotp
\end{eqnarray*}
Finally $B(\beta)$ admits a derivative:
\begin{equation}
B'(\beta) = \beta^{-2}\int_{R^2}\frac{1}{L(z)P}p_H((\beta L(z)P)^{-1})d z \cdotp
\end{equation}
This concludes the proof.
\end{proof}

The CDF and PDF of $\xi_m$ are easily derived according to the property of Poisson point process:
\begin{lemma}\label{proposition CDF PDF of a_ms}
 The complementary cumulative distribution function of
$\xi_m$ is given by:
\begin{equation}
P(\xi_m>t) =
e^{-\lambda_BB(t)}\sum_{i=0}^{m}\frac{(\lambda_BB(t))^i}{i!},
\end{equation}
and its probability density function is given by
\begin{equation}\label{equ: PEF mth si}
p_{\xi_m}(t) = \lambda_B^{m+1}B'(t)B(t)^me^{-\lambda_BB(t)} \cdotp
\end{equation}
\end{lemma}
\begin{IEEEproof}
The event "$\xi_m>t$" is equivalent to the event "in the
interval $[0,t]$ there is at most $m$ points" and the number
of points in this interval follows a Poisson random variable of
mean $\lambda_BB(t)$, so:
\begin{eqnarray*}P(\xi_m<t) =
e^{-\lambda_BB(t)}\sum_{i=0}^{m}(\lambda_BB(t))^i
\end{eqnarray*}
The PDF is thus given by $p_{\xi_m}(t) = -\frac{d}{dt}P(\xi_m<t)$, and after some simple manipulations we obtain the equation (\ref{equ: PEF mth si})
\end{IEEEproof}
\subsection{Special cases}
In this section we derive closed forms for $B(\beta)$ in some special cases.
\paragraph{Path loss exponent model}
\begin{lemma}
If $L(z)=K|z|^{-\gamma}$ then:
\begin{equation}
B(\beta) = C.\beta^{\frac{2}{\gamma}} ,
\end{equation}
where $C=\pi(PK)^{\frac{2}{\gamma}}E(H^{\frac{2}{\gamma}})$.
\end{lemma}
\begin{proof}
We have:
\begin{eqnarray*}
B(\beta) &=& 2\pi\int_0^{\infty}r1_{\{tPK\beta\geq
r^{\gamma}\}}p_H(t)d r d t\\
        &=& 2\pi\int_0^{\infty}p_H(t)d t
        \int_0^{(tKP\beta)^{1/\gamma}}rd r\\
        &=&\pi(PK)^{\frac{2}{\gamma}}\beta^{\frac{2}{\gamma}}\int_0^{\infty}p_H(t)t^{\frac{2}{\gamma}}d t\\
        &=&\pi(PK)^{\frac{2}{\gamma}}E(H^{\frac{2}{\gamma}})\beta^{\frac{2}{\gamma}} \cdotp
\end{eqnarray*}
Hence the result.
\end{proof}
Remark that this result can be derived from \cite{AGeometricInterpretationofFadinginWirelessNetworks}. We observe that the distribution of the point process $\Xi$ does depend only on $E(H^{\frac{2}{\gamma}})$ but not on the distribution of shadowing $H$ itself. This phenomenon can be explained as in \cite{InterferenceinLargeWirelessNetworks}(page 159).
\paragraph{Modified path loss exponent model}
\begin{lemma}
If $L(z)=K(\max\{R_0,|z|\})^{-\gamma}$ then:
\begin{equation}\label{equ: corr: modif exponent}
B(\beta)=C_1\beta^{\frac{2}{\gamma}}\int_{\frac{
        R_0^{\gamma}}{\beta PK}}^{\infty}t^{\frac{2}{\gamma}}p_H(t)d t ,
\end{equation}
where $C_1=\pi (PK)^{\frac{2}{\gamma}}$. In addition, we have:
\begin{equation} \label{equ: corr: modif exponent path Bphay}
B'(\beta) = \frac{2}{\gamma}\beta^{-1}B(\beta)+\pi R_0^2p_H(\frac{R_0^{\gamma}}{PK\beta}) \cdotp
\end{equation}
If the slow fading is lognormal shadowing $H\sim
10^{G/10}$ where $G\sim \mathcal{N}(0,\sigma^2)$ we have:
\begin{equation}
B(\beta) =
C_1\beta^{\frac{2}{\gamma}}e^{(\frac{2\sigma_1}{\gamma})^2}Q(\frac{-\ln\beta-\ln(PKR_0^{-\gamma})}{\sigma_1}-\frac{2\sigma_1}{\gamma})
\end{equation}
where $Q(u)=\frac{1}{\sqrt{2\pi}}\int_u^{\infty}e^{-u^2/2} du$ is the \emph{Q}-function and $\sigma_1=\frac{\sigma\ln 10 }{10}$.
\end{lemma}
\begin{IEEEproof}
Similarly to the pathloss exponent model case, we have:
\begin{eqnarray*}
B(\beta) &=& 2\pi\int_{R^2}rF_H((\max\{R_0,r\})^{-\gamma}(PK\beta)^{-1})d r\\
        &=& \pi R_0^2F_H( R_0^{\gamma}(PK\beta)^{-1}) +\\
        && + 2\pi\int_{\frac{
        R_0^{\gamma}}{\beta PK}}^{\infty}p_H(t)d t
        \int_{R_0}^{(tKP\beta)^{1/\gamma}}rd r\\
        &=&C_1\beta^{\frac{2}{\gamma}}\int_{\frac{
        R_0^{\gamma}}{PK\beta}}^{\infty}t^{\frac{2}{\gamma}}p_H(t)d t \cdotp
\end{eqnarray*}
We obtain the equation (\ref{equ: corr: modif exponent}). Derivative two sides of that equation and do some simple manipulations we obtain the equation (\ref{equ: corr: modif exponent path Bphay}). In the case of lognormal shadowing we have:
\begin{eqnarray*}
B(\beta) &=&
C_1\beta^{\frac{2}{\gamma}}\int_{\frac{
        R_0^{\gamma}}{PK\beta}}^{\infty}\frac{1}{\sqrt{2\pi\sigma_1^2}t}t^{\frac{2}{\gamma}}e^{-\frac{(\ln t)^2}{2\sigma_1^2}}d t\\
        &=& C_1\beta^{\frac{2}{\gamma}}\int_{\ln\frac{
        R_0^{\gamma}}{PK\beta}}^{\infty}\frac{1}{\sqrt{2\pi\sigma_1^2}}e^{\frac{2u}{\gamma}}e^{-\frac{u^2}{2\sigma_1^2}}d
        u\\
        &=&C_1\beta^{\frac{2}{\gamma}}e^{(\frac{2\sigma_1}{\gamma})^2}\int_{\ln\frac{
        R_0^{\gamma}}{PK\beta}}^{\infty}\frac{1}{\sqrt{2\pi\sigma_1^2}}e^{-\frac{(u-\frac{2\sigma_1^2}{\gamma})^2}{2\sigma_1^2}}d
        u \cdotp
\end{eqnarray*}
Here the results.
\end{IEEEproof}
As a consequence, both the exponent path loss model and its
modified model satisfy the assumption \ref{assumption: Max is obtian}.

\section{Outage analysis}\label{Section: Outage proba}
\subsection{General case}
Here, we remark that the outage probability and the coverage
probability do not depend on the time index.
So we can drop the time slot parameter in this section. The expression of the SINR can be rewritten as:
\begin{equation}
s_x =
\frac{r_{y_{0}x}\xi_0^{-1}}{N+\sum_{i\neq
0}1_{\{e_i=e_{0}\}}a(\theta_i)r_{y_ix}\xi_i^{-1}}\cdotp
\end{equation}
The outage probability is calculated as below:

\begin{theorem}\label{theorem: Outage general}The outage probability is given by
\begin{eqnarray*}\label{equ: theorem : outage general}
p_o=1-\lambda_B\int_0^{\infty}B'(\beta)e^{-\lambda_BB(\beta)-NT\mu\beta-\frac{\lambda_B}{2\pi k}D(\beta)}d\beta
\end{eqnarray*}
where  $D(\beta)=\int_{-\pi}^{\pi}d\theta\int_{\beta}^{\infty}B'(\xi)\frac{ d \xi }{1+\xi(T\beta a(\theta))^{-1}} $.
\end{theorem}
\begin{IEEEproof}
To calculate the
outage probability $P(s_x<T)$, we will calculate the coverage
probability $P(s_x\geq T)$.

We first consider the conditional probability $P(s_x\geq T|\xi_0=\beta)$.
Because $r_{y_{0}x}$ is an exponential random variable of mean
$1/\mu$ we have:
\begin{eqnarray*}
P(s_x\geq T|\xi_0=\beta)&=& P(r_{y_0x}\geq T\beta(N+I_x(\beta))|\xi_0=\beta)\\
        &=&E(e^{-\mu T\beta(N+I_x(\beta))}|\xi_0=\beta)\\
        &=& e^{-NT\mu\beta}\mathcal{L}_{I_{x}(\beta)}(T\mu\beta)
\end{eqnarray*}
where  $I_{x}(\beta)$
is the distribution of the random variable $I_x$ given on the event $(\xi_0=\beta))$ and
$\mathcal{L}_{I_{x}(\beta)}$ is its Laplace transform.  Conditioning
on the event $(\xi_0=\beta)$ the point process
$\{\xi_i\}_{i>0}$ is a Poisson point process on $(\beta,\infty)$
with intensity $\lambda_BB'(\xi)d \xi$ according to the strong Markov
property. By thinning theorem, the point process
$\{\xi_i\}_{\{i>0,e_i=e_{0}\}}$ is a
Poisson point process on $(0,\beta)$ with intensity $\frac{\lambda_B}{k}B'(\xi)d \xi$. Hence,
$\mathcal{L}_{I_{x}(\beta)}$ can be calculated as follows (\cite{StochasticGeometryandWirelessNetworksVolumeI}, shot noise theory):
\begin{eqnarray*}
\mathcal{L}_{I_{x}(\beta)}(u) &=& e^{-\int_{\beta}^{\infty}
\frac{\lambda_B}{2\pi k}B'(\xi) (1-E(e^{-a(\theta)u\xi^{-1} R}))d \xi}\\
            &=& e^{-\frac{\lambda_B}{2\pi k}\int_{\beta}^{\infty} B'(\xi)d \xi
\int_0^{\infty}d r \int_{-\pi}^{\pi}\mu
e^{-\mu r}(1-e^{-a(\theta)ur\xi^{-1}}) d \theta }\\
            &=&e^{-\frac{\lambda_B}{2\pi k}\int_{-\pi}^{\pi}d\theta\int_{\beta}^{\infty}B'(\xi)\frac{ d \xi}{1+\xi\mu(u a(\theta))^{-1}}} \cdotp
\end{eqnarray*}
We get that:
\begin{eqnarray*}
P(s_x\geq T|\xi_0=\beta)=\\
=e^{-NT\mu\beta-\frac{\lambda_B}{2\pi k}\int_{-\pi}^{\pi}d\theta\int_{\beta}^{\infty}B'(\xi)\frac{ d \xi }{1+\xi(T\beta .a(\theta))^{-1}} },
\end{eqnarray*}
thus
\begin{equation}  \label{equ : conditional outage}
= e^{-NT\mu\beta-\frac{\lambda_B}{2\pi k}D(\beta)} \cdotp
\end{equation}
Since the distribution density of $\xi_0$  is $\lambda_BB'(\beta)e^{-\lambda_BB(\beta)}$ (proposition \ref{proposition CDF PDF of a_ms}), by averaging over all $\xi_0$ we obtain the equation (\ref{equ: theorem : outage general}).

\end{IEEEproof}

\subsection{Special cases}
\paragraph{Interference limited}
\begin{collary}
In the interference-limited regime, we have

\begin{eqnarray}\label{equ: coll : outage interference limited}
p_o(T)=1-\lambda_B\int_0^{\infty}B'(\beta)e^{-\lambda_BB(\beta)-\frac{\lambda_B}{2\pi k}D(\beta)}d\beta \cdotp
\end{eqnarray}

\end{collary}

\paragraph{Path loss exponent model}
\begin{collary}\label{collary: Outage proba pathloss exponent}
If $L(z)=K|z|^{-\gamma}$ we have:
\begin{equation}\label{equ : collary: outage path ex}
p_o(T) = 1-\int_0^{\infty}e^{-M\alpha-G\alpha^{\frac{\gamma}{2}}}d\alpha
\end{equation}
where $M:=M(k,T,\gamma)=1+\frac{1}{2\pi k}\int_{-\pi}^{\pi}d\theta\int_1^{\infty}\frac{ d
u}{1+(T.a(\theta))^{-1}u^{\frac{\gamma}{2}}}$ and $G=NT\mu(\lambda_BC)^{-\frac{\gamma}{2}}$.
\end{collary}
\begin{IEEEproof}
Since $B(\xi)=C.\xi^{\frac{2}{\gamma}}$ and
$B'(\xi)=\frac{2C}{\gamma}\xi^{\frac{2}{\gamma}-1}$ we have:
\begin{eqnarray*}
D(\beta) &=&   \int_{-\pi}^{\pi}d \xi\int_{\beta}^{\infty}\frac{2C}{\gamma}\xi^{\frac{2}{\gamma}-1}\frac{d\theta }{1+\xi(T\beta a(\theta))^{-1}}                  \\
  &=&           C.\beta^{\frac{2}{\gamma}}\int_{-\pi}^{\pi}d\theta\int_{\beta}^{\infty}\frac{ d (\frac{\xi}{\beta})^{\frac{2}{\gamma}}}{1+\frac{\xi}{\beta}(T a(\theta))^{-1}}\\
        &=& C.\beta^{\frac{2}{\gamma}}\int_{-\pi}^{\pi}d\theta\int_{1}^{\infty}\frac{du }{1+(T a(\theta))^{-1}u^{\frac{\gamma}{2}}} \cdotp
\end{eqnarray*}
Plug it into (\ref{equ: theorem : outage general}) we have:
\begin{eqnarray*}
p_c(T)&=&\int_0^{\infty}\frac{2\lambda_BC}{\gamma}\beta^{\frac{2}{\gamma}-1}e^{-\lambda_BCM\beta^{\frac{2}{\gamma}}-NT\mu\beta}d\beta\\
        &=& \int_0^{\infty}e^{-M\alpha-G\alpha^{\frac{\gamma}{2}}}d\alpha \cdotp
\end{eqnarray*}
\end{IEEEproof}
Remark that if $\gamma = 4$, we can find that:
\begin{eqnarray*}
M &=& 1 + \frac{1}{2\pi k}\int_{-\pi}^{\pi}d\theta\int_1^{\infty}\frac{d u }{1+(Ta(\theta))^{-1}u^{2}}\\
    &=& 1 + \frac{1}{2\pi k}\int_{-\pi}^{\pi}\sqrt{Ta(\theta)}(\frac{\pi}{2}-\arctan\frac{1}{\sqrt{Ta(\theta)}})d\theta
\end{eqnarray*}
and
\begin{eqnarray*}
p_o(T) &=&  1 - e^{\frac{M^2}{4G}}\int_0^{\infty}e^{-(\sqrt{G}\alpha+\frac{M}{2\sqrt{G}})^2} d \alpha
\\&=&1 - \frac{\sqrt{2\pi}}{G}e^{\frac{M^2}{4G}}Q(\frac{M}{2\sqrt{G}}) \cdotp
\end{eqnarray*}
\paragraph{Interference limited and path loss exponent model} In this case, the outage probability is easily derived from (\ref{equ : collary: outage path ex}) by setting $N=0$.
\begin{collary}
If $L(z)=K|z|^{-\gamma}$ and $N=0$ we
have:
\begin{equation}
p_o(T) = 1-\frac{1}{M} \cdotp
\end{equation}
\end{collary}

\subsection{Observations and interpretations}
Some interesting facts are observed from above results:

\begin{itemize}

\item Rewrite the expression of SINR as
\begin{eqnarray*}
s_x[l]
&=&   \frac{\overline{r}_{y_0x}[l]\xi_0^{-1}}{\mu N+\sum_{i\neq
0}1_{\{e_i=e_0\}}a(\theta_i)\overline{r}_{y_ix}[l] \xi_i^{-1}}
\end{eqnarray*}
where $\overline{r}_{y_0x}[l]=\mu r_{y_ix}[l]$. Since $r_{yx}[l]$ is an exponential random variable of mean $1/\mu$, $\overline{r}_{y_0x}[l]$ is an exponential random variable of mean $1$. Hence by the above equation it is expected that the outage probability depends on the product $\mu N$ but not directly on $\mu$ and $N$. It is increasing function of $N \mu$which is confirmed by the equation (\ref{equ: theorem : outage general}). The fact that the outage probability is the increasing function of $\mu$ and $N$ is quite natural, the increase of noise or the degrade of the channel fast fading always makes the system work worsts.

\item It is also expected that in the interference limited case ($N=0$) the outage probability does not depend on $\mu$. It is confirmed by the equation (\ref{equ: coll : outage interference limited}). Physically it means that in the absence of noise, the fast fading increases or degrades the channels to the MS of the serving BS and the interfering BS at the same level, thus the SINR will not change.

\item In the interference limited and exponent path loss model case, the outage probability does not depend on $\mu$, the BS density $\lambda_B$, or the distribution of shadowing $H$. It is due to the scaling property of the exponent path loss and the homogeneous Poisson point process. The outage probability is a decreasing function of the path loss exponent $\gamma$, reflecting the fact that bad propagation environment degrades the received SINR.

\item In the presence of noise $N>0$ and exponent path loss model case, the outage probability is a increasing function of $\lambda_B$. Hence, it can be thought that the more an operator installs BSs, the better the network is. In addition, if the density of BSs goes to infinite then outage will never occur. However it is not true. In fact, if the density of BSs is very high, the distance between a MS and its serving BS and some interfering BSs is relatively close. Here, the exponent path loss model is no longer valid since it is not accurate at small distance. If the modified exponent path loss is used which is more appreciate, the outage probability must converge to $0$. The outage probability is also a increasing function of $E(H^{\frac{2}{\gamma}})$, and if the shadowing $H$ follows lognormal distribution then the outage probability will be increasing function of $\sigma$. We recover an other well known fact: the increase of uncertainty of the radio channel degrades the performance of the network.

\end{itemize}

\section{Handover analysis} \label{section: Handover analysis}
\subsection{General case}
If the MS is in outage in $n$ consecutive time slots, a handover decision has to be made. Keep in mind that only the Rayleigh fast fading changes each time slot, and the other network components do not change. Let $A_l$
be the event that the mobile being in outage in the time slot $l$,
and $A_l^c$ its complement and observe that in fact
$P(\cap_{i=1}^mA_{j_i}^c)=P(\cap_{i=1}^mA_{i}^c)$. By definition $p_{ho} := P(\cap_{i=1}^nA_{l+i-1}) = P(\cap_{i=1}^nA_i)$. We have
\begin{eqnarray*}
p_{ho} &=& 1+\sum_{m=1}^n(-1)^m\sum_{j_1\neq...\neq j_m\in
\{1,..,n\}}P(\cap_{i=1}^mA_{j_i}^c)\\
        &=& 1+\sum_{m=1}^n(-1)^m\frac{n!}{m!(n-m)!}P(\cap_{i=1}^mA_i^c) \cdotp
\end{eqnarray*}
\begin{theorem}
The handover probability is given by:
\begin{eqnarray*}
p_{ho} = 1+\sum_{m=1}^n(-1)^m\frac{n!}{m!(n-m)!}q_m,
\end{eqnarray*}
where $q_m=P(\cap_{i=1}^mA_i^c)$ is given by:
\begin{eqnarray*}
q_m= \int_0^{\infty}\lambda_BB'(\beta)e^{-\lambda_BB(\beta)-NT\mu\beta-\frac{\lambda_B}{2\pi k}D_m(\beta)}d\beta,
\end{eqnarray*}
and $D_m(\beta) = \int_{-\pi}^{\pi}d\theta\int_{\beta}^{\infty}B'(\xi)(1-(\frac{1}{1+T\beta a(\theta)\xi^{-1}})^m) d \xi$.
\end{theorem}
\begin{IEEEproof}

We need to calculate the probability
$P(\cap_{i=1}^mA_i^c)$ that is the probability that the mobile
is covered in $m$ different time slots. The calculation is similar
the that in section \ref{Section: Outage proba}. We begin with
calculating the conditional probability
$P(\cap_{i=1}^mA_i^c|\xi_0=\beta)$:
\begin{eqnarray*}
P(\cap_{i=1}^mA_i^c|\xi_0=\beta) = P( s_x[1]\geq T,..., s_x[m]\geq T|\xi_0=\beta)\\
        =  P(r_{y_0x}[i]\geq \beta (TN+I_{x}[i])i=1..m|\xi_0=\beta)\\
        =  E(e^{-\mu (mTN\beta+\sum_{i=1}^mI_{x}(\beta)[i])}|\xi_0=\beta)\\
        = e^{-mNT\mu\beta}\mathcal{L}_{\sum_{i=1}^mI_x(\beta)[i]}(T\mu\beta)
\end{eqnarray*}
where $I_x(\beta)[i]$ is the distribution of the random variable $I_x[i]$ given $(\xi_0=\beta)$. We have :
\begin{eqnarray*}
\sum_{i=1}^mI_x(\beta)[i] = \sum_{j=1}^{\infty}1_{\{e_i = e_0\}}\xi_i^{-1}a(\theta_i)(\sum_{i=1}^mr_{y_ix}[i])\cdotp
\end{eqnarray*}
As the random variables $r_{y_ix}[i]$ are independent copies of the exponential random variable $R$, the random variables $\sum_{i=1}^mr_{y_ix}[i]$ are also i.i.d and the common Laplace transform of the later are :
\begin{eqnarray*}
\mathcal{L}_{\sum_{i=1}^mr_{y_ix}[i]}(u) &=&  (\mathcal{L}_{R}(u))^m\\
                                &=& (\frac{\mu}{\mu+u})^m \cdotp
\end{eqnarray*}
The Laplace transform of $\sum_{i=1}^mI_x(\beta)[i]$ is now:
\begin{eqnarray*}
\mathcal{L}_{\sum_{i=1}^mI_x(\beta)[i]}(u) = e^{-\frac{\lambda_B}{2\pi k}\int_{-\pi}^{\pi}d\theta\int_{\beta}^{\infty}B'(\xi)(1-(\frac{\mu}{\mu+a(\theta)\xi^{-1}u})^m) d \xi} \cdotp
\end{eqnarray*}
The conditional probability is then given by:
\begin{eqnarray*}
P(\cap_{i=1}^mA_i^c|\xi_0=\beta)  =
e^{-mNT\mu\beta-\frac{\lambda_B}{2\pi k}D_m(x)}\cdotp
\end{eqnarray*}
By averaging with respect to $\xi_0$, we have:
\begin{eqnarray*}
q_m= \int_0^{\infty}\lambda_BB'(\beta)e^{-\lambda_BB(\beta)-NT\mu\beta-\frac{\lambda_B}{2\pi k}D_m(\beta)}d\beta \cdotp
\end{eqnarray*}
\end{IEEEproof}
This concludes the proof.
\subsection{Special cases}
We can obtain more closed expression for $q_m$ in some special cases.
\paragraph{Interference limited}
\begin{collary}
In the interference limited regime $N=0$, we have:
\begin{eqnarray*}
q_m= \int_0^{\infty}\lambda_BB'(\beta)e^{-\lambda_BB(\beta)-\frac{\lambda_B}{2\pi k}D_m(\beta)}d\beta \cdotp
\end{eqnarray*}
\end{collary}
\paragraph{Path loss exponent model}
\begin{collary}
If $L(z)=K|z|^{-\gamma}$ then:
\begin{eqnarray*}
q_m = \int_0^{\infty}e^{-M_m\alpha-G\alpha^{\frac{\gamma}{2}}}d\alpha
\end{eqnarray*}
where $M_m = 1 + \frac{1}{2\pi k}\int_{-\pi}^{\pi}d\theta\int_1^{\infty}(1-(\frac{1}{1+Ta(\theta)u^{-\frac{\gamma}{2}}})^m)du$.
\end{collary}
\begin{IEEEproof}
The proof follows the same lines as the proof of \ref{collary: Outage proba pathloss exponent}
\end{IEEEproof}
Closed expression of $q_m$ is obtained in the case $\gamma = 4$:
\begin{eqnarray*}
q_m = \frac{\sqrt{2\pi}}{G}e^{\frac{M_m^2}{4G}}Q(\frac{M_m}{2\sqrt{G}}) \cdotp
\end{eqnarray*}
\paragraph{Interference limited and path loss exponent model}
\begin{collary}
If $N=0$ and $L(z)=K|z|^{-\gamma}$ we have:
\begin{eqnarray*}
q_m = \frac{1}{M_m} \cdotp
\end{eqnarray*}
\end{collary}

\subsection{Observations and interpretations}
Some interesting facts are observed from above results and they are similar to the properties of outage probability:
\begin{itemize}
\item The handover probability is increasing function of $N\mu$.

\item In the interference limited and exponent path loss model case, the handover probability does not depend on $\mu$, the BSs density $\lambda_B$, nor the distribution of shadowing $H$. The handover probability is a decreasing function of the path loss exponent $\gamma$.

\item In the presence of noise $N>0$ and exponent path loss model case, the handover probability is a increasing function of $\lambda_B$. Thus, in this case the more an operator installs BSs, the less a MS has to do handover. But it is not true in a real system. As previously explained, in the case of very dense BSs, the pathloss exponent model is no longer accurate  for small distance. The handover probability is also a increasing function of $E(H^{\frac{2}{\gamma}})$, therefore if the shadowing $H$ is lognormal shadowing the handover probability will be increasing function of $\sigma$.

\end{itemize}
\section{Numerical results and comparison to the hexagonal model}\label{section: Simulation}
We place a MS at the origin $o$ and consider a region $B(o,R_g)$ where $R_g=10.000(m)$. The BSs are distributed as a Poisson point process in this region. The path loss exponent model is considered. The default values of model are placed on the table \ref{Table: Defaut parameters value}. They are not changed throughout the simulation.

In literature, the hexagonal model is widely used and studied so we would like to compare two models. For a fair comparison, the density of BSs must be chosen to be the same, i.e the area of a hexagonal cell will be $1/\lambda_B$. Unlike the Poisson model where each BS is randomly assigned a frequency, in the hexagonal model, the frequencies are well assigned so that an interfering BS is far from the transmitting BS and BSs of different frequency are grouped in reuse patterns. The reuse factor $k$ in the hexagonal model is determined by $k=i^2+j^2+ij$ where integers $i,j$ are the relative location of co-channel cell. The MS is uniformly chosen on the surface of the center cell. The same signal propagation model and the scenario described in the section \ref{subsec: Scenario} are applied in the hexagonal model.

The figure \ref{figure outage proba vs SINR thershold} shows the outage probability versus the SINR threshold of the Poisson model and the hexagonal model in the case $k=7$. As we can see, the outage probability in the case of Poisson model is always greater than that of hexagonal model which is intuitive. The different is about $8$ (dB) in the case $\gamma =4$ and $6$(dB) in the case $\gamma = 3$.

In the figure \ref{figure outage proba vs gam} plotted the outage probability as a function of $\gamma$. We can see that the outage probability is a decreasing function of $\gamma$ as theoretically observed. In the figure \ref{figure handover proba vs gam} plotted the handover probability versus the SINR threshold of Poisson model and hexagonal model. If the reuse factor $k$ increase, the MS has to do less handover. Thus, increase the reuse factor has a positive effect on the system performance not only in term of outage but also in term of handover.

\begin{table}[!t]
\renewcommand{\arraystretch}{1.3}
\caption{Model parameters' default values}
\label{Table: Defaut parameters value}
\centering
\begin{tabular}{|c|c|c|c|}
\hline
$K$ & $P$& $n_t$ & $\mu$\\
\hline
 $-20$ (dB) & $0$(dBm) & $8$ &$1$ \\
\hline
\end{tabular}
\end{table}

\begin{figure}[!t]
\centering
\includegraphics[width=3.5in]{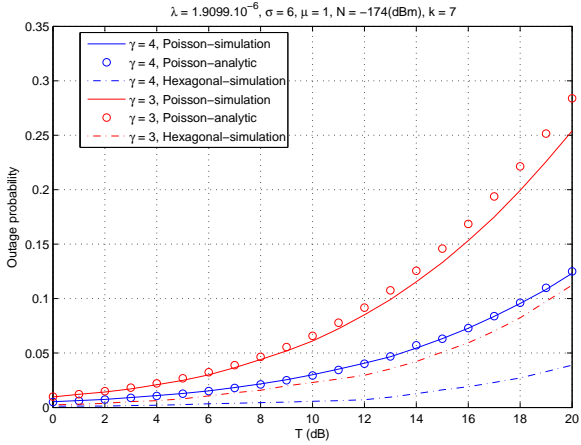}
\caption{Outage probability vs SINR threshold}
\label{figure outage proba vs SINR thershold}
\end{figure}

\begin{figure}[!t]
\centering
\includegraphics[width=3.5in]{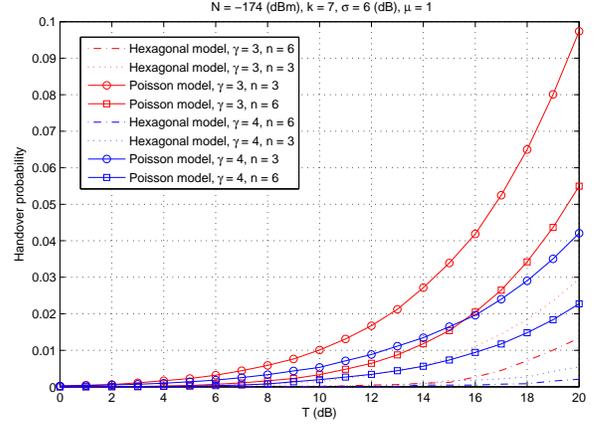}
\caption{Handover probability vs SINR threshold}
\label{figure handover proba}
\end{figure}

\begin{figure}[!t]
\centering
\includegraphics[width=3.5in]{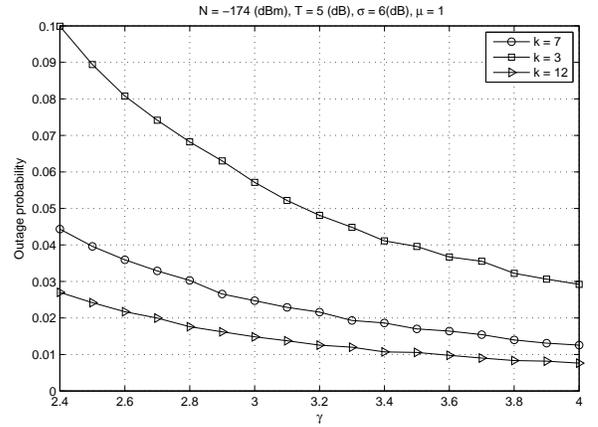}
\caption{Outage probability vs path loss exponent $\gamma$, Poisson model}
\label{figure outage proba vs gam}
\end{figure}

\begin{figure}[!t]
\centering
\includegraphics[width=3.5in]{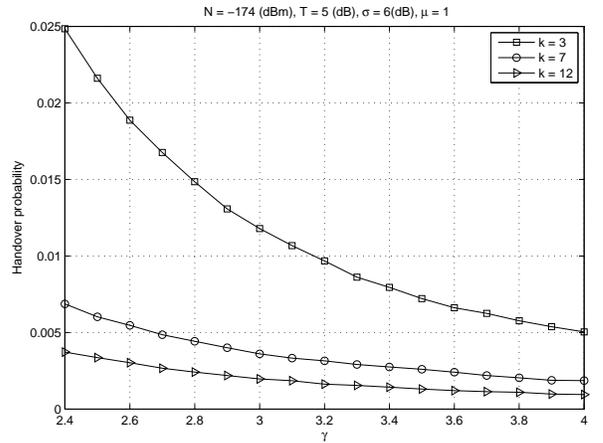}
\caption{Handover probability vs path loss exponent $\gamma$, Poisson model, $n=3$}
\label{figure handover proba vs gam}
\end{figure}

\section{Conclusion}
In this paper we have investigated the outage and handover probabilities of wireless cellular networks taking into account the reuse factor, the beamforming, the path loss, the slow fading and the fast fading. We valid our model by simulation and compare numerical results to that of hexagonal model. The analytical expressions derived in the this paper can be considered as an upper bound for a real system.

\end{document}